\documentclass[10pt, reqno]{amsart}
\usepackage{graphicx, amssymb, amsmath, amsthm}
\numberwithin{equation}{section}

\def\pa{\partial}

\let\Re=\undefined\DeclareMathOperator*{\Re}{Re}
\let\Im=\undefined\DeclareMathOperator*{\Im}{Im}

\newcommand{\R}{\mathbb{R}}
\newcommand{\C}{\mathbb{C}}

\newcommand{\eps}{\varepsilon}

\newcommand{\la}{\mathcal{L}_a}

\newtheorem{theorem}{Theorem}[section]

\newtheorem{lemma}[theorem]{Lemma}

\newtheorem{proposition}[theorem]{Proposition}

\theoremstyle{definition}
\newtheorem{definition}[theorem]{Definition}
\newtheorem{remark}[theorem]{Remark}

\newcommand{\qtq}[1]{\quad\text{#1}\quad}

\makeatletter
\newcommand{\Extend}[5]{\ext@arrow0099{\arrowfill@#1#2#3}{#4}{#5}}
\makeatother

\begin{document}
\title[Focusing NLS]{Focusing NLS with inverse square potential}

%
%
\author[J. Zheng]{Jiqiang Zheng}
\address{Institute of Applied Physics and Computational Mathematics, Beijing 100088, China}
\email{zhengjiqiang@gmail.com}

\begin{abstract}
In this paper, we utilize the method in \cite{BM} to establish the
radial scattering result for the focusing nonlinear Schr\"odinger
equation with inverse square potential $i\pa_tu-\la u=-|u|^{p-1}u$
in the energy space $H^1_a(\R^d)$ in dimensions $d\geq3$, which
extends the result of \cite{KMVZ,LMM} to higher dimensions cases but
with radial initial data. The new ingredient is to establish the
dispersive estimate for radial function and overcome the weak
dispersive estimate when $a<0$.

\end{abstract}

 \maketitle

\begin{center}
 \begin{minipage}{100mm}
 \end{minipage}
 \end{center}


\section{Introduction}

\noindent We study the initial-value problem for focusing
nonlinear
 Schr\"odinger equations of the form
\begin{align} \label{equ1.1}
\begin{cases}    (i\partial_t-\la)u= -|u|^{p-1}u,\quad
(t,x)\in\R\times\R^d,
\\
u(0,x)=u_0(x)\in H^1(\R^d),
\end{cases}
\end{align}
where $u:\R_t\times\R_x^d\to \C$ and $\la=-\Delta+\frac{a}{|x|^2}$.

The class of solutions to \eqref{equ1.1} is left invariant by the
scaling
\begin{equation}\label{scale}
u(t,x)\mapsto \lambda^{\frac2{p-1}}u(\lambda^2t, \lambda
x),\quad\lambda>0.
\end{equation}
Moreover, one can also check that the only homogeneous $L_x^2$-based
Sobolev space that is left invariant under \eqref{scale} is
$\dot{H}_x^{s_c}(\R^d)$ with $s_c:=\tfrac{d}2-\tfrac2{p-1}$.
Solutions to \eqref{equ1.1} conserve their \emph{mass} and
\emph{energy} by
\begin{align*}
& M(u(t)) := \int_{\R^d} |u(t,x)|^2 \,dx, \\
& E_a(u(t)) := \int_{\R^d} \tfrac12|\nabla u(t,x)|^2 +
\tfrac{a}{2|x|^2} |u(t,x)|^2 - \tfrac1{p+1} |u(t,x)|^{p+1} \,dx.
\end{align*}
Initial data belonging to $H_x^1(\R^d)$ have finite mass and energy.
This follows from equivalent of Sobolev norm  and the following
variant of the Gagliardo-Nirenberg inequality:
\begin{equation}\label{E:GN}
\|f\|_{L_x^{p+1}(\R^d)}^{p+1} \leq C_a
\|f\|_{L_x^2(\R^d)}^\frac{d+2-(d-2)p}2
\|\sqrt{\la}f\|_{L_x^2(\R^d)}^\frac{d(p-1)}2,
\end{equation}
where $C_a$ denotes the sharp constant in the inequality above for
radial functions. We will show in Theorem \ref{T:GN} that the sharp
constant $C_{a}$ is attained by a radial solution $Q_{a}$ to
elliptic equation $-\la Q_a-Q_a+Q_a^p=0$.

 The functions $Q_{a}$ provide examples of non-scattering solutions at the radial threshold via $u(t,x) = e^{it}Q_{a}(x)$.
 We consider the problem of global existence and scattering for \eqref{equ1.1} below threshold.  We begin with the following definitions.

\begin{definition}[Solution, scattering]\label{D:solution} Let $t_0\in\R$ and $u_0\in H_x^1(\R^d)$. Let $I$ be an interval containing $t_0$.
A function $u:I\times\R^d\to\C$ is a \emph{solution} to
\eqref{equ1.1}, if it belongs to $C_t H_a^1\cap L_t^5
H_a^{1,\frac{10d}{5d-4}}(K\times\R^d)$ for any compact $K\subset I$
and obeys the Duhamel formula
\[
u(t) = e^{-i(t-t_0)\la}u_0 + i\int_{t_0}^t e^{-i(t-s)\la}\bigl(|u(s)|^{p-1} u(s)\bigr)\,ds\qtq{for all}t\in I,
\]
where we rely on the self-adjointness of $\la$ to make sense of $e^{-it\la}$ via the Hilbert space functional calculus.
 We call $I$ the \emph{lifespan} of $u$. We call $u$ a \emph{maximal-lifespan solution} if it cannot be extended to any strictly larger interval.
  If $I=\R$, we call $u$ \emph{global}.

Moreover, a global solution $u$ to \eqref{equ1.1}  \emph{scatters}
if there exist $u_\pm\in H_x^1(\R^d)$ such that
\[
\lim_{t\to\pm\infty} \| u(t) - e^{-it\la}u_{\pm} \|_{H_x^1(\R^d)} =
0.
\]
\end{definition}

In this paper, we utilize the method in \cite{BM} to obtain the following threshold result for the class of radial solutions:

\begin{theorem}[Radial scattering/blowup dichotomy]\label{T:radial} Let $(a,d,p)$ satisfy
\begin{equation}\label{equ:acond}
a>\begin{cases}-\big(\frac{d-2}2\big)^2\quad \text{if}\quad d=3\quad
\text{and}\quad \frac43<p-1\leq2\\
-\big(\frac{d-2}2\big)^2+\big(\frac{d-2}2-\frac1{p-1}\big)^2\quad
\text{if}\quad d\geq3\quad \text{and}\quad
\frac2{d-2}\vee\frac4d<p-1<\frac4{d-2},
\end{cases}
\end{equation}
where $a\vee b:=\max\{a,b\}.$ Let $u_0\in H_x^1(\R^d)$ be radial and
satisfy $M(u_0)^{1-s_c}E_a(u_0)^{s_c}<M(Q_a)^{1-s_c}E_a(Q_a)^{s_c}$.
Moreover, if $$\|u_0\|_{L_x^2}^{1-s_c} \|u_0\|_{\dot H_a^1}^{s_c}
<\|Q_a\|_{L_x^2}^{1-s_c} \|Q_a\|_{\dot H_a^1}^{s_c},$$ then the
solution to \eqref{equ1.1} with initial data $u_0$ is global and
scatters.

\end{theorem}

\begin{remark}
$(i)$ In the case $a=0$, such result was firstly considered by
Holmer and Roudenko\cite{HR} for the  3D cubic radial  and
Duyckaerts-Holmer-Roudenko\cite{DHR} for nonradial data. Lately,
Killip, Murphy, Visan and the third author \cite{KMVZ} and Lu, Miao
and Murphy \cite{LMM} generalized their result to the focusing
Schr\"odinger equation with inverse square potential, i.e.
\eqref{equ1.1}. In this paper, we extend the result of
\cite{KMVZ,LMM} to general nonlinear term in dimensions $d\geq3$ but
with radial initial data. We also refer the reader to the defocusing
nonlinear Schr\"odinger equation with inverse square potential
\cite{KMVZZ2,ZZ}. The main new ingredient of this paper is to
establish the dispersive estimate for radial function and overcome
the weak dispersive estimate when $a<0$.

$(ii)$ The restriction on $(a,d,p)$ stems from the local
well-posedness theory in $H^1(\R^d)$ for \eqref{equ1.1}. While in
the proof of local well-posedness, we need to estimate powers of
$\mathcal{L}_a$ applied to the nonlinearity term.   To obtain the
requisite fractional calculus estimates for $\mathcal{L}_a$, we rely
on the equivalence of Sobolev spaces  to exchange powers of
$\mathcal{L}_a$ and powers of $-\Delta$ (for which fractional
calculus estimates are known). This argument leads to a restriction
on the range of $(a,d,p)$ as in \eqref{equ:acond}.


\end{remark}

We sketch the idea and argument for the proof here. First, by variational analysis and blowup criterion, we derive that the solution $u$ is global.
  And then, by radial Sobolev embedding and dispersive estimate, we establish a scattering criterion as the case $a=0$ \cite{Tao}. Here we should be careful in the case $a<0$, since we have only the weak dispersive estimate, see Lemma \ref{thm:disp}.  Finally, using Virial argument, radial Sobolev embedding and  variational analysis, we prove the above  scattering criterion.

We conclude the introduction by giving some notations which
will be used throughout this paper. To simplify the expression of
our inequalities, we introduce some symbols $\lesssim, \thicksim,
\ll$. If $X, Y$ are nonnegative quantities, we use $X\lesssim Y $ or
$X=O(Y)$ to denote the estimate $X\leq CY$ for some $C$, and $X
\thicksim Y$ to denote the estimate $X\lesssim Y\lesssim X$. We use
$X\ll Y$ to mean $X \leq c Y$ for some small constant $c$. We use
$C\gg1$ to denote various large finite constants, and $0< c \ll 1$
to denote various small constants. For any $r, 1\leq r \leq \infty$,
we denote by $\|\cdot \|_{r}$ the norm in
$L^{r}=L^{r}(\mathbb{R}^3)$ and by $r'$ the conjugate exponent
defined by $\frac{1}{r} + \frac{1}{r'}=1$.




\section{Preliminaries}

\subsection{Harmonic analysis for $\la$} In this section, we collect some harmonic analysis tools adapted to the operator $\la$.
 The primary reference for this section is \cite{KMVZZ1}.

For $1< r < \infty$, we write $\dot H^{1,r}_a(\R^d)$ and $
H^{1,r}_a(\R^d)$ for the homogeneous and inhomogeneous Sobolev
spaces associated with $\la$, respectively, which have norms
$$
\|f\|_{\dot H^{1,r}_a(\R^d)}= \|\sqrt{\la} f\|_{L^r(\R^d)} \qtq{and}
\|f\|_{H^{1,r}_a(\R^d)}= \|\sqrt{1+ \la} f\|_{L^r(\R^d)}.
$$
When $r=2$, we simply write $\dot H^{1}_a(\R^d)=\dot
H^{1,2}_a(\R^d)$ and $H^{1}_a(\R^d)=H^{1,2}_a(\R^d)$.

 By the sharp Hardy inequality, the operator $\la$ is positive precisely for $a\geq -(\frac{d-2}2)^2$.
Denote
\begin{equation}\label{equ:sigma}
\sigma:=\tfrac{d-2}2-\bigr[\bigl(\tfrac{d-2}2\bigr)^2+a\bigr]^{\frac12}.
\end{equation}

Estimates on the heat kernel associated to the operator $\mathcal{L}_a$ were found by Liskevich--Sobol \cite{LS} and Milman--Semenov \cite{MS}.

\begin{lemma}[Heat kernel bounds, \cite{LS, MS}] \label{L:kernel}  Let $d\geq 3$ and $a\geq -(\tfrac{d-2}{2})^2$.
 There exist positive constants $C_1,C_2$ and $c_1,c_2$ such that for any $t>0$ and any $x,y\in\R^d\backslash\{0\}$,
\[
C_1(1\vee\tfrac{\sqrt{t}}{|x|})^\sigma(1\vee\tfrac{\sqrt{t}}{|y|})^\sigma t^{-\frac{d}{2}} e^{-\frac{|x-y|^2}{c_1t}} \leq e^{-t\la}(x,y) \leq
C_2(1\vee\tfrac{\sqrt{t}}{|x|})^\sigma(1\vee\tfrac{\sqrt{t}}{|y|})^\sigma t^{-\frac{d}{2}} e^{-\frac{|x-y|^2}{c_2t}}.
\]
\end{lemma}
As a consequence, we can obtain the following equivalence of Sobolev
spaces.
\begin{lemma}[Equivalence of Sobolev spaces, \cite{KMVZZ1}]\label{pro:equivsobolev} Let $d\geq 3$, $a\geq -(\frac{d-2}{2})^2$, and $0<s<2$. If $1<p<\infty$ satisfies $\frac{s+\sigma}{d}<\frac{1}{p}< \min\{1,\frac{d-\sigma}{d}\}$, then
\[
\||\nabla|^s f \|_{L_x^p}\lesssim_{d,p,s} \|(\la)^{\frac{s}{2}} f\|_{L_x^p}\qtq{for all}f\in C_c^\infty(\R^d\backslash\{0\}).
\]
If $\max\{\frac{s}{d},\frac{\sigma}{d}\}<\frac{1}{p}<\min\{1,\frac{d-\sigma}{d}\}$, then
\[
\|(\la)^{\frac{s}{2}} f\|_{L_x^p}\lesssim_{d,p,s} \||\nabla|^s f\|_{L_x^p} \qtq{for all} f\in C_c^\infty(\R^d\backslash\{0\}).
\]
\end{lemma}

We will make use of the following fractional calculus estimates due to Christ and Weinstein \cite{CW}.  Combining these estimates with Lemma~\ref{pro:equivsobolev}, we can deduce analogous statements for the operator $\la$ (for restricted sets of exponents).

\begin{lemma}[Fractional calculus]\text{ }
\begin{itemize}
\item[(i)] Let $s\geq 0$ and $1<r,r_j,q_j<\infty$ satisfy $\tfrac{1}{r}=\tfrac{1}{r_j}+\tfrac{1}{q_j}$ for $j=1,2$. Then
\[
\| |\nabla|^s(fg) \|_{L_x^r} \lesssim \|f\|_{L_x^{r_1}} \||\nabla|^s g\|_{L_x^{q_1}} + \| |\nabla|^s f\|_{L_x^{r_2}} \| g\|_{L_x^{q_2}}.
\]
\item[(ii)] Let $G\in C^1(\C)$ and $s\in (0,1]$, and let $1<r_1\leq \infty$  and $1<r,r_2<\infty$ satisfy $\tfrac{1}{r}=\tfrac{1}{r_1}+\tfrac{1}{r_2}$. Then
\[
\| |\nabla|^s G(u)\|_{L_x^r} \lesssim \|G'(u)\|_{L_x^{r_1}} \|u\|_{L_x^{r_2}}.
\]
\end{itemize}
\end{lemma}
We will need the following radial Sobolev embedding from \cite{Tao}.

\begin{lemma}[Radial Sobolev embedding]\label{lem:sobemb}
Let $d\geq3.$ For radial $f\in H^1(\R^d)$, there holds
\begin{equation}\label{equ:radsobemb}
\big\||x|^sf\big\|_{L_x^\infty(\R^d)}\lesssim\|f\|_{H^1(\R^d)},
\end{equation}
for $\tfrac{d}2-1\leq s\leq\tfrac{d-1}2.$
\end{lemma}

 Let $f$ be Schwartz function defined on $\R^d$, we define
the Hankel transform of order $\nu$:
\begin{equation}\label{2.14}
(\mathcal{H}_{\nu}f)(\xi)=\int_0^\infty(r\rho)^{-\frac{d-2}2}J_{\nu}(r\rho)f(r\omega)r^{d-1}\mathrm{d}r,
\end{equation}
where $\rho=|\xi|$, $\omega=\xi/|\xi|$ and $J_{\nu}$ is the Bessel
function of order $\nu$  defined by the
integral
\begin{equation*}
J_\nu(r)=\frac{(r/2)^\nu}{\Gamma(\nu+\frac12)\Gamma(1/2)}\int_{-1}^{1}e^{isr}(1-s^2)^{(2\nu-1)/2}\mathrm{d
}s\quad\text{with}~ \nu>-\frac12~\text{and}~ r>0.
\end{equation*} Specially, if the function $f$ is radial,
then
\begin{equation}\label{2.15}
(\mathcal{H}_{\nu}f)(\rho)=\int_0^\infty(r\rho)^{-\frac{d-2}2}J_{\nu}(r\rho)f(r)r^{d-1}\mathrm{d}r.
\end{equation}

The following properties of the Hankel transform are obtained in
\cite{BPSTZ}:
\begin{lemma}\label{Hankel}
Let $\mathcal{H}_{\nu}$ be defined above and $
A_{\nu}:=-\partial_r^2-\frac{d-1}r\partial_r+\big[\nu^2-\big(\frac{d-2}2\big)^2\big]{r^{-2}}.
$ Then

$(\rm{i})$ $\mathcal{H}_{\nu}=\mathcal{H}_{\nu}^{-1}$,

$(\rm{ii})$ $\mathcal{H}_{\nu}$ is self-adjoint, i.e.
$\mathcal{H}_{\nu}=\mathcal{H}_{\nu}^*$,

$(\rm{iii})$ $\mathcal{H}_{\nu}$ is an $L^2$ isometry, i.e.
$\|\mathcal{H}_{\nu}\phi\|_{L^2_\xi}=\|\phi\|_{L^2_x}$,

$(\rm{iv})$ $\mathcal{H}_{\nu}(
A_{\nu}\phi)(\xi)=|\xi|^2(\mathcal{H}_{\nu} \phi)(\xi)$, for
$\phi\in L^2$.
\end{lemma}

 \subsection{Strichartz estimates and dispersive estimate}\label{sze}

Strichartz estimates for the propagator $e^{-it\la}$  were proved by
Burq, Planchon, Stalker, and Tahvildar-Zadeh in \cite{BPSTZ}.
Combining these with the Christ--Kiselev Lemma \cite{CK}, we obtain
the following Strichartz estimates:

\begin{proposition}[Strichartz estimate, \cite{BPSTZ,ZZ1}] Let $d\geq3$, and fix $a>-\big(\tfrac{d-2}2\big)^2$. The solution $u$ to $(i\partial_t-\la)u = F$
 on an interval $I\ni t_0$ obeys
\[
\|u\|_{L_t^q L_x^r(I\times\R^d)} \lesssim \|u(t_0)\|_{L_x^2(\R^d)} +
\|F\|_{L_t^{\tilde q'} L_x^{\tilde r'}(I\times\R^d)}
\]
for any $2\leq q,\tilde q\leq\infty$ with
$\frac{2}{q}+\frac{d}{r}=\frac{d}{\tilde q}+\frac{3}{\tilde r}=
\frac{d}2$ including $(q,\tilde q)= (2,2)$. \end{proposition}

As a consequence of Strichartz estimate, we obtain the local
well-posedness theory in $H^1(\R^d)$.

\begin{theorem}[Local well-posedness, \cite{KMVZ,LMM}]\label{T:LWP} Let $(a,d,p)$ satisfy the condition \eqref{equ:acond}. Assume
 $u_0\in H_x^1(\R^d)$, and $t_0\in\R$.  Then the following hold:

\begin{itemize}
\item[(i)] There exist $T=T(\|u_0\|_{H_a^1})>0$ and a unique solution $u:(t_0-T,t_0+T)\times\R^d\to\C$ to \eqref{equ1.1}
 with $u(t_0)=u_0$. In particular, if $u$ remains uniformly bounded in $H_a^1$ throughout its lifespan, then $u$ extends to a global  solution.
\item[(ii)]  There exists $\eta_0>0$ such that if
\[
\|e^{-i(t-t_0)\mathcal{L}_a}u_0\|_{L_{t,x}^{\frac{d+2}2(p-1)}((t_0,\infty)\times\R^d)}
< \eta\qtq{for some} 0 <\eta<\eta_0,
\]
then the solution $u$ to \eqref{equ1.1} with data $u(t_0)=u_0$ is
forward-global and satisfies
\[
\|u\|_{L_{t,x}^{\frac{d+2}2(p-1)}((t_0,\infty)\times\R^d)} \lesssim
\eta.
\]
The analogous statement holds backward in time (as well as on all of
$\R$).
\item[(iii)] For any $\psi\in H_a^1$, there exist $T>0$ and a solution $u:(T,\infty)\times\R^d\to\C$ to \eqref{equ1.1} such that
\[
\lim_{t\to\infty} \|u(t) - e^{-it\mathcal{L}_a}\psi\|_{H_a^1} = 0.
\]
The analogous statement holds backward in time.
\end{itemize}
\end{theorem}

Next, we prove the key estimate (dispersive estimate) which will be
useful in the proof of scattering criterion (Lemma \ref{lem:scatcri}
below).

\begin{theorem}[Dispersive estimate]\label{thm:disp} Let $f$ be radial function.

$(i)$ If $a\geq0$, then we have
\begin{equation}\label{equ:disp}
\|e^{it\la}f\|_{L^\infty(\R^d)}\leq C|t|^{-\frac{d}2}\|f\|_{L_x^1(\R^d)}.
\end{equation}

$(ii)$ If $-\tfrac{(d-2)^2}4<a<0$, then there holds
\begin{equation}\label{equ:dispaleq0}
\big\|(1+|x|^{-\sigma})^{-1}e^{it\la}f\big\|_{L^\infty(\R^d)}\leq C\frac{1+|t|^{\sigma}}{|t|^{\frac{d}2}}\big\|(1+|x|^{-\sigma})f\big\|_{L_x^1(\R^d)},
\end{equation}
with $\sigma$ being as in \eqref{equ:sigma}.
\end{theorem}

\begin{proof}
Since $f(x)$ is radial, $u(t,x):=e^{it\la}f$ solve
\begin{equation}\label{equ:utfxrd}
\begin{cases}
i\pa_tu-A_{\nu}u=0,\\
u(0,r)=f(r),
\end{cases}
\end{equation}
where the operator $A_\nu$ is defined as in Lemma \ref{Hankel} with $\nu=\frac{d-2}2-\sigma.$
Applying the Hankel transform to the equation \eqref{equ:utfxrd}, by
$(\rm{iv})$ in Lemma \ref{Hankel}, we have
\begin{equation}\label{3.6}
\begin{cases}
i\partial_{t}
\tilde{ u}-\rho^2\tilde{u}=0 \\
\tilde{u}(0,\rho)=(\mathcal{H}_{\nu}f)(\rho),
\end{cases}
\end{equation}
where
$\tilde{u}(t,\rho)=(\mathcal{H}_{\nu}
u)(t,\rho)$.
Solving this ODE and inverting the Hankel transform, we obtain
\begin{align*}
u(t,r)=&\int_0^\infty (r\rho)^{-\frac{d-2}2}J_\nu(r\rho)e^{-it\rho^2}(\mathcal{H}_\nu f)(\rho)\rho^{d-1}\;d\rho\\
=&\int_0^\infty (r\rho)^{-\frac{d-2}2}J_\nu(r\rho)e^{-it\rho^2}\rho^{d-1}\int_0^\infty (s\rho)^{-\frac{d-2}2}J_\nu(s\rho)f(s)s^{d-1}\;ds\;d\rho\\
=&\int_0^\infty f(s)s^{d-1}K(t,r,s)\;ds,
\end{align*}
with the kernel
\begin{align*}
K(t,r,s)=&(rs)^{-\frac{d-2}2}\int_0^\infty J_\nu(r\rho)J_\nu(s\rho)e^{-it\rho^2}\rho\;d\rho\\
=&(rs)^{-\frac{d-2}2}\frac{e^{-\frac12\nu\pi i}}{2it}e^{-\frac{r^2+s^2}{4it}}J_\nu\Big(\frac{rs}{2t}\Big),
\end{align*}
where we used the analytic continuation as in \cite{Ford} in the second equality.
Thus,
\begin{align*}
u(t,r)=&\frac{e^{-\frac12\nu\pi i}}{2it}\int_0^\infty f(s)s^{d-1}(rs)^{-\frac{d-2}2}e^{-\frac{r^2+s^2}{4it}}J_\nu\Big(\frac{rs}{2t}\Big)\;ds\\
=&\frac{e^{-\frac12\nu\pi i}}{2it}\Big(\int_0^\frac{2t}{r}+\int_\frac{2t}r^\infty\Big)\Big(f(s)s^{d-1}(rs)^{-\frac{d-2}2}e^{-\frac{r^2+s^2}{4it}}J_\nu\Big(\frac{rs}{2t}\Big)
\Big)\;ds\\
\triangleq&I+II.
\end{align*}
Using $|J_\nu(r)|\lesssim r^{-\frac12}$ with $r\geq1$, we obtain
\begin{align*}
|II|\leq&Ct^{-\frac{d}2}\int_\frac{2t}r^\infty|f(s)|s^{d-1}\big(\tfrac{rs}{2t}\big)^{-\frac{d-1}2}\;ds\leq Ct^{-\frac{d}2}\|f\|_{L_x^1(\R^d)}.
\end{align*}
On the other hand, by $|J_\nu(r)|\lesssim r^{\nu}$ with $r\leq1$, we get for $a\geq0$
\begin{align*}
|I|\leq&Ct^{-1}\int_0^\frac{2t}r|f(s)|s^{d-1}(rs)^{-\frac{d-2}2}\Big(\frac{rs}{2t}\Big)^\nu\;ds\\
\leq&Ct^{-\frac{d}2}\int_0^\frac{2t}r|f(s)|s^{d-1}\Big(\frac{rs}{2t}\Big)^{-\sigma}\;ds\\
\leq& Ct^{-\frac{d}2}\|f\|_{L_x^1(\R^d)},
\end{align*}
while for $a<0$
\begin{align*}
(1+r^{-\sigma})^{-1}|I|\leq&Ct^{-\frac{d}2}\int_0^\frac{2t}r(1+s^{-\sigma})|f(s)|s^{d-1}\frac{t^\sigma}{(1+r^\sigma)(1+s^\sigma)}\;ds\\
\leq&Ct^{-\frac{d}2}t^\sigma\big\|(1+|x|^{-\sigma})f\big\|_{L_x^1(\R^d)}.
\end{align*}
Therefore by collecting all of them, we conclude the proof of Theorem \ref{thm:disp}.

\end{proof}

\section{Variational analysis}\label{S:var}

In this section, we carry out the variational analysis for the sharp Gagliardo--Nirenberg inequality, which leads naturally to the thresholds appearing in Theorem~\ref{T:radial}.

\begin{theorem}[Sharp Gagliardo--Nirenberg inequality]\label{T:GN} Fix $a>-\frac{(d-2)^2}4$ and define
\[
C_a := \sup\bigl\{ \|f\|_{L_x^{p+1}}^{p+1} \div
\bigl[\|f\|_{L_x^2}^{\frac{d+2-(d-2)p}2} \|f\|_{\dot
H_a^1}^{\frac{d(p-1)}2}\bigr]:f\in H_a^1\backslash\{0\},~f~{\rm
radial}\bigl\}.
\]
Then $C_a\in(0, \infty)$ and the  Gagliardo--Nirenberg inequality for radial functions
\begin{equation}\label{E:GN}
\|f\|_{L_x^{p+1}}^{p+1}\leq C_a\|f\|_{L_x^2}^{\frac{d+2-(d-2)p}2}
\|f\|_{\dot H_a^1}^{\frac{d(p-1)}2}
\end{equation}
 is attained by a function $Q_a\in H_a^1$, which is a non-zero, non-negative, radial solution to the elliptic problem
\begin{equation}
\label{elliptic}
-\la Q_a - Q_a + Q_a^p = 0.
\end{equation}

\end{theorem}

\begin{proof}   Define the functional
\[
J_a(f) :=
\frac{\|f\|_{L_x^{p+1}}^{p+1}}{\|f\|_{L_x^2}^{\frac{d+2-(d-2)p}2}
\|f\|_{\dot H_a^1}^{\frac{d(p-1)}2}}, \qtq{so that} C_a =
\sup\{J_a(f):f\in H_a^1\backslash\{0\},~f~{\rm radial}\}.
\]
Note that the standard Gagliardo--Nirenberg inequality and the equivalence of Sobolev spaces imply  $0<C_a<\infty$.

We prove by mimicking the well-known proof for $a=0$ and Theorem 3.1
in \cite{KMVZ}. Take the sequence of radial functions
$\{f_n\}\subset H_a^1\backslash\{0\}$ such that $J_a(f_n)\nearrow
C_a$. Choose $\mu_n\in\R$ and $\lambda_n\in\R$ so that
$g_n(x):=\mu_n f_n(\lambda_nx)$ satisfy
$\|g_n\|_{L_x^2}=\|g_n\|_{\dot H_a^1}=1$. Note that
$J_a(f_n)=J_a(g_n)$.
 As $H^1_{\rm rad}\hookrightarrow L^{p+1}_x$ compactly, passing to a subsequence we may assume
 that $g_n$ converges to some $g\in H_a^1$ strongly in $L_x^{p+1}$ as well as weakly in $H_a^1$.
  As $g_n$ is an optimizing sequence, we deduce that $C_a=\|g\|_{L_x^{p+1}}^{p+1}$.
  We also have that $\|g\|_{L_x^2} = \|g\|_{\dot H_a^1} = 1$, or else $g$ would be a super-optimizer. Thus $g$ is an optimizer.

The Euler--Lagrange equation for $g$ is given by
\[
-\frac{d(p-1)}2C_a\la g -\frac{d+2-(d-2)p}2 C_a g + (p+1)g^p = 0.
\]
Thus, if we define $Q_a$ via
\[
g(x) = \alpha Q_a(\lambda x), \qtq{with} \alpha =\Big(
\tfrac{d+2-(d-2)p}{2(p+1)}C_a\Big)^\frac1{p-1} \qtq{and} \lambda =
\sqrt{\tfrac{d+2-(d-2)p}{3(p-1)}},
\]
then $Q_a$ is an optimizer of \eqref{E:GN} that solves \eqref{elliptic}.

\end{proof}

By integration by part,  we easily get

\begin{lemma}\label{Pohozaev}
Assume $\phi\in\mathcal{S}(\mathbb{R}^d)$, then
\begin{align*}
\int_{\mathbb{R}^d} \Delta\phi
x\cdot\nabla\phi dx&=\frac{d-2}{2}\int_{\mathbb{R}^d}|\nabla\phi|^2dx,\\
\int_{\mathbb{R}^d}\phi x\cdot\nabla\phi dx&=-\frac{d}{2}\int_{\mathbb{R}^d}|\phi|^2dx,\\
\int_{\mathbb{R}^d}x\cdot\nabla\phi \phi^p
dx&=-\frac{d}{p+1}\int_{\mathbb{R}^d}|\phi|^{p+1}dx.
\end{align*}
\end{lemma}

By a simple computation, we have
\begin{lemma}\label{lem:qaident}
Let $Q_a$ be the solution to $-\la Q_a-Q_a+Q_a^p=0$. Then
\begin{equation}\label{equ:identme}
\|Q_a\|_{L_x^2}^2=\frac{d+2-(d-2)p}{2(p+1)}\|Q_a\|_{L_x^{p+1}}^{p+1},~\|Q_a\|_{\dot{H}^1_a}^2=\frac{d(p-1)}{2(p+1)}\|Q_a\|_{L_x^{p+1}}^{p+1},
\end{equation}
and
\begin{equation}\label{equ:idenenery}
E_a(Q_a)=\frac{dp-(d+4)}{2d(p-1)}\|Q_a\|_{\dot{H}^1_a}^2=\frac{dp-(d+4)}{4(p+1)}\|Q_a\|_{L_x^{p+1}}^{p+1}.
\end{equation}
Moreover,
\begin{equation}\label{equ:caid}
C_a=\frac{\|Q_a\|_{L_x^{p+1}}^{p+1}}{\|Q_a\|_{L_x^2}^{\frac{d+2-(d-2)p}2}
\|Q_a\|_{\dot
H_a^1}^{\frac{d(p-1)}2}}=\Big(\frac{d+2-(d-2)p}{2(p+1)}\Big)^{-\frac{d+2-(d-2)p}4}
\Big(\frac{d(p-1)}{2(p+1)}\Big)^{-\frac{d(p-1)}4}\|Q_a\|_{L_x^{p+1}}^{\frac{(d-1)(p^2-1)}2}.
\end{equation}
and
\begin{equation}\label{equ:qah1}
C_a\|Q_a\|_{L^2}^{(1-s_c)(p-1)}\|Q_a\|_{\dot{H}^1_a}^{s_c(p-1)}=\frac{2(p+1)}{d(p-1)}.
\end{equation}

\end{lemma}

\begin{proposition}[Coercivity]\label{P:coercive} Fix $a>-\frac{(d-2)^2}4$. Let $u:I\times\R^d\to\C$
be the maximal-lifespan solution to \eqref{equ1.1} with $u(t_0)=u_0\in H_a^1\backslash\{0\}$ for some $t_0\in I$. Assume that
\begin{equation}\label{quant-below}
M(u_0)^{1-s_c}E_a(u_0)^{s_c} \leq (1-\delta)M(Q_a)^{1-s_c}E_a(Q_a)^{s_c}\qtq{for some}\delta>0.
\end{equation}
Then there exist $\delta'=\delta'(\delta)>0$, $c=c(\delta,a,\|u_0\|_{L_x^2})>0$, and $\eps=\eps(\delta)>0$ such that:
If $\|u_0\|_{L_x^2}^{1-s_c} \| u_0\|_{\dot H_a^1}^{s_c} \leq \|Q_a\|_{L_x^2}^{1-s_c} \|Q_a\|_{\dot H_a^1}^{s_c} $, then for all $t\in I$,
\begin{itemize}
\item[(i)] $\|u(t)\|_{L_x^2}^{1-s_c} \|u(t)\|_{\dot H_a^1}^{s_c} \leq (1-\delta')\|Q_a\|_{L_x^2}^{1-s_c} \|Q_a\|_{\dot H_a^1}^{s_c}$,
\item[(ii)] $\|u(t)\|_{\dot H_a^1}^2 - \tfrac{d(p-1)}{2(p+1)} \|u(t)\|_{L_x^{p+1}}^{p+1} \geq c\|u(t)\|_{\dot H_a^1}^2$
\item[(iii)] $(\tfrac{dp-(d+4)}{2d(p-1)}+\tfrac{2\delta'}{3(p-1)}) \|u(t)\|_{\dot H_a^1}^2 \leq E_a(u) \leq \tfrac12 \| u(t)\|_{\dot H_a^1}^2,$
\end{itemize}

\end{proposition}

\begin{proof}   By the sharp Gagliardo--Nirenberg inequality, conservation of mass and energy, and \eqref{quant-below}, we may write
\begin{align*}
(1-\delta)M(Q_a)^{1-s_c}E_a(Q_a)^{s_c} \geq& M(u)^{1-s_c} E_a(u)^{s_c}\\
\geq& \|u(t)\|_{L_x^2}^{2(1-s_c)}\Big(\frac12\|u(t)\|_{\dot
H_a^1}^2-\frac1{p+1}C_a\|u(t)\|_{L_x^2}^{\frac{d+2-(d-2)p}2}
\|u(t)\|_{\dot H_a^1}^{\frac{d(p-1)}2}\Big)^{s_c}
\end{align*}
for any $t\in I$. Using \eqref{equ:idenenery} and \eqref{equ:qah1}, this inequality becomes
\[
(1-\delta)^\frac1{s_c}\geq
\frac{d(p-1)}{dp-(d+4)}\biggl(\frac{\|u(t)\|_{L_x^2}^{1-s_c}\|u(t)\|_{\dot
H_a^1}^{s_c}}{\|Q_a\|_{L_x^2}^{1-s_c} \|Q_a\|_{\dot
H_a^1}^{s_c}}\biggr)^\frac2{s_c} -
\frac2{dp-(d+4)}\biggl(\frac{\|u(t)\|_{L_x^2}^{1-s_c} \|u(t)\|_{\dot
H_a^1}^{s_c}}{\|Q_a\|_{L_x^2}^{1-s_c} \|Q_a\|_{\dot
H_a^1}^{s_c}}\biggr)^{\frac2{s_c}(p-1)}.
\]
Claims (i) now follow from a continuity argument, together with the observation that
\[
(1-\delta)^\frac1{s_c} \geq \frac{d(p-1)}{dp-(d+4)}y^\frac2{s_c} -
\frac2{dp-(d+4)}y^{\frac2{s_c}(p-1)} \implies |y-1|\geq \delta'
\qtq{for some} \delta'=\delta'(\delta)>0.
\]

For claim (iii), the upper bound follows immediately, since the nonlinearity is focusing.  For the lower bound, we again rely on the sharp Gagliardo--Nirenberg. Using (i)  and \eqref{equ:qah1} as well, we find
\begin{align*}
E_a(u) &  \geq \tfrac12\|u(t)\|_{\dot H_a^1}^2 [ 1 - \tfrac2{p+1} C_a \|u(t)\|_{L_x^2}^{(1-s_c)(p-1)} \|u(t) \|_{\dot H_a^1}^{s_c(p-1)} ] \\
& \geq  \tfrac12\|u(t)\|_{\dot H_a^1}^2 [ 1 - \tfrac4{d(p-1)}
(1-\delta')^{p-1}] \geq
(\tfrac{dp-(d+4)}{2d(p-1)}+\tfrac{2\delta'}{3(p-1)})\|u(t)\|_{\dot
H_a^1}^2
\end{align*}
for all $t\in I$. Thus (iii) holds.

We turn to (ii). We begin by writing
\begin{align*}
\|u(t)\|_{\dot H_a^1}^2 - \tfrac{d(p-1)}{2(p+1)} \|u(t)\|_{L_x^{p+1}}^{p+1}& = \tfrac{d(p-1)}2E_a(u) - \tfrac{dp-(d+4)}4 \|u(t)\|_{\dot H_a^1}^2, \\
(1+\eps)\|u(t)\|_{\dot H_a^1}^2 - \tfrac{d(p-1)}{2(p+1)}
\|u(t)\|_{L_x^{p+1}}^{p+1}&= \tfrac{d(p-1)}2E_a(u) -
(\tfrac{dp-(d+4)}4-\eps)\|u(t)\|_{\dot H_a^1}^2,
\end{align*}
for $t\in I$. Thus (ii) follows from (iii) by choosing any $0<c\leq \delta'$.

  \end{proof}

\begin{remark}\label{R:coercive} Suppose $u_0\in H_x^1\backslash\{0\}$ satisfies $M(u_0)^{1-s_c}E_a(u_0)^{s_c}<M(Q_a)^{1-s_c}E_a(Q_a)^{s_c}$ and $\|u_0\|_{L_x^2}^{1-s_c} \|u_0\|_{\dot H_a^1}^{s_c} \leq \|Q_a\|_{L_x^2}^{1-s_c} \|Q_a\|_{\dot H_a^1}^{s_c}.$ Then by continuity, the maximal-lifespan solution $u$ to \eqref{equ1.1} with initial data $u_0$ obeys $\|u(t)\|_{L_x^2}^{1-s_c} \|u(t)\|_{\dot H_a^1}^{s_c} < \|Q_a\|_{L_x^2}^{1-s_c} \|Q_a\|_{\dot H_a^1}^{s_c}$ for all $t$ in the lifespan of $u$.  In particular, $u$ remains bounded in $H_x^1$ and hence is global.
\end{remark}



\section{Proof of Theorem \ref{T:radial}}

In this section, we turn to prove Theorem \ref{T:radial}. Assume that $u$ is a solution to \eqref{equ1.1} satisfying the hypotheses of Theorem \ref{T:radial}. It follows from Remark \ref{R:coercive} that $u$ is global and satisfies the uniform bound
\begin{equation}\label{equ:uniformboun}
\|u_0\|_{L_x^2}^{1-s_c}\|u(t)\|_{\dot{H}^1_a}^{s_c}<(1-\delta')\|Q_a\|_{L_x^2}^{1-s_c}\|Q_a\|_{\dot{H}^1_a}^{s_c}.
\end{equation}

To show Theorem \ref{T:radial}, we first establish a scattering criterion by following the argument as in \cite{BM,Tao}.

\subsection{scattering criterion}

\begin{lemma}[Scattering criterion]\label{lem:scatcri}
Suppose $u:~\R\times\R^d\to\mathbb{C}$ is a radial solution to
\eqref{equ1.1} satisfying
\begin{equation}\label{equ:uhasuni}
\|u\|_{L_t^\infty(\R,H^1(\R^d))}\leq E.
\end{equation}
There exist $\epsilon=\epsilon(E)>0$ and $R=R(E)>0$ such that if
\begin{equation}\label{equ:masslim}
\liminf\limits_{t\to\infty}\int_{|x|<R}|u(t,x)|^2\;dx\leq\epsilon^2,
\end{equation}
then, $u$ scatters forward in time.
\end{lemma}

\begin{proof}
First, by interpolation with $\|u\|_{L_t^\infty(\R,H^1(\R^d))}$, we
only need to show that
\begin{equation}\label{equ:reduce}
\|u\|_{L_t^{4}([0,\infty),L_x^{r}(\R^d))}<+\infty,
\end{equation}
with $r=\tfrac{2d}{d-2}$. By H\"older's inequality, Sobolev
embedding and \eqref{equ:uhasuni}, we have for any finite interval
$I$,
\begin{equation*}
\|u\|_{L_t^{4}(I,L_x^{r})}\leq C|I|^\frac14\|u\|_{L_t^\infty
\dot{H}^1}\leq C|I|^\frac14.
\end{equation*}
 Thus, we are
reduced to show for some $T>0$
\begin{equation}\label{equ:redfur}
\|u\|_{L_t^{4}([T,\infty),L_x^{r}(\R^d))}<+\infty.
\end{equation}
By  continuity argument, Strichartz estimate and Sobolev embedding, we are further reduced to show
\begin{equation}\label{equ:redfur}
\|e^{i(t-T)\la}u(T)\|_{L_t^{4}([T,\infty),L_x^{r}(\R^d))}\ll1.
\end{equation}

Now, let $0<\epsilon<1$ and $R\geq1$ to be determined later. Using
Duhamel formula, we can write
\begin{equation}\label{equ:duhamp}
e^{i(t-T)\la}u(T)=e^{it\la}u_0+F_1(t)+F_2(t),
\end{equation}
where
$$F_j(t)=i\int_{I_j}e^{i(t-s)\la}(|u|^{p-1}u)(s)\;ds,~I_1=[0,T-\epsilon^{-\theta}],~I_2=[T-\epsilon^{-\theta},T],$$
where $0<\theta<1$ to be determined later. Using Sobolev embedding,
Strichartz estimate, we can pick $T_0$ sufficiently large such that
\begin{equation}\label{equ:smallline}
\big\|e^{it\la}u_0\big\|_{L_t^{4}([T_0,\infty),L_x^{r})}<\epsilon.
\end{equation}

{\bf Estimate the term $F_1(t)$:} We can rewrite $F_1(t)$ as
\begin{equation}\label{equ:f2rw}
F_1(t)=e^{i(t-T+\epsilon^{-\theta})\la}\big[u(T-\epsilon^{-\theta})\big]-e^{it\la}u_0.
\end{equation}
Using Strichartz estimate, we have
\begin{equation}\label{equ:l4l6est}
\|F_1\|_{L_t^2([T,\infty),L_x^r)}\lesssim1.
\end{equation}
On the other hand, by Lemma \ref{thm:disp}, we get for $p\geq2$
\begin{align}\nonumber
\|F_1(t)\|_{L_x^r(|x|\leq
R_1)}\lesssim&\int_{I_1}\big\|e^{i(t-s)\la}(|u|^{p-1}u)(s)\big\|_{L_x^r(|x|\leq
R_1)}\;ds\\\nonumber
\lesssim&\int_{I_1}\big\|(1+|x|^{-\alpha_1})^{-1}e^{i(t-s)\la}(|u|^{p-1}u)(s)\big\|_{L_x^\infty}\;ds
\cdot\|(1+|x|^{-\alpha_1})\|_{L_x^\frac{2d}{d-2}(|x|\leq
R_1)}\\\nonumber
\lesssim&R_1^\frac{d-2}2\int_{I_1}|t-s|^{-\frac{d}2+\alpha_1}\big\|(1+|x|^{-\alpha_1})(|u|^{p-1}u)\big\|_{L_x^1}\;ds\\\nonumber
\lesssim&R_1^\frac{d-2}2|t-T+\epsilon^{-\theta}|^{-\frac{d-2}2+\alpha_1}\\\label{equ:pgeq2}
\lesssim&R_1^\frac{d-2}2\epsilon^{(\frac{d-2}2-\alpha_1)\theta},\quad
t>T,
\end{align}
where
\begin{equation}\label{equ:alpha1}
\alpha_1=\begin{cases} 0\quad \text{if}\quad a\geq0\\
\sigma\quad\text{if}\quad a<0,
\end{cases}
\end{equation}
and we have used the estimate for $a\geq0$
$$\big\||u|^{p-1}u\big\|_{L^1_x}\leq\|u\|_{L_x^p}^p\leq
C\|u\|_{L_t^\infty H^1_x}<+\infty,$$ while for $a<0$
\begin{align*}
\big\||x|^{-\sigma}|u|^{p-1}u\big\|_{L^1_x}\lesssim\big\||x|^{-\sigma}u\big\|_{L_x^2}\|u\|_{L_x^{2(p-1)}}^{p-1}
\lesssim\big\||\nabla|^\sigma
u\big\|_{L_x^2}\|u\|_{H^1_x}^{p-1}\lesssim\|u\|_{H^1_x}^p<+\infty,
\end{align*}
since $\sigma<1$ by the assumption \eqref{equ:acond}. When $p<2$, we
have
\begin{align}\nonumber
\|F_1(t)\|_{L_x^r(|x|\leq
R_1)}\lesssim&\int_{I_2}\big\|e^{i(t-s)\la}(|u|^{p-1}u)(s)\big\|_{L_x^r(|x|\leq
R_1)}\;ds\\\nonumber
\lesssim&\int_{I_2}\big\|(1+|x|^{-\alpha_1})^{-(p-1)}e^{i(t-s)\la}(|u|^{p-1}u)(s)\big\|_{L_x^\frac2{2-p}}\;ds\\\nonumber
&\times\|(1+|x|^{-\alpha_1})^{p-1}\|_{L_x^\frac{2d}{pd-(d+2)}(|x|\leq
R_1)}\\\nonumber
\lesssim&R_1^\frac{pd-(d+2)}2\int_{I_2}|t-s|^{-(\frac{d}2+\alpha_1)(p-1)}\big\|(1+|x|^{-\alpha_1})^{p-1}(|u|^{p-1}u)\big\|_{L_x^\frac2p}\;ds\\\nonumber
\lesssim&R_1^\frac{pd-(d+2)}2|t-T+\epsilon^{-\theta}|^{-(\frac{d}2+\alpha_1)(p-1)+1}\\\label{equ:pleq2}
\lesssim&R_1^\frac{pd-(d+2)}2\epsilon^{(\frac{d}2+\alpha_1)(p-1)\theta-\theta},\quad
t>T,
\end{align}
where $\alpha_1$ is as in \eqref{equ:alpha1} and we have used the
estimate for $a\geq0$
$$\big\||u|^{p-1}u\big\|_{L_x^\frac2p}\leq\|u\|_{L_x^2}^p<+\infty$$
and for $a<0$
\begin{align*}
\big\||x|^{-(p-1)\sigma}|u|^{p-1}u\big\|_{L_x^\frac2p}\lesssim\big\||x|^{-\sigma}u\big\|_{L_x^2}^{p-1}\|u\|_{L_x^2}\lesssim\|u\|_{L_t^\infty
H^1_x}^p<+\infty.
\end{align*}

Using \eqref{equ:f2rw}, Lemma \ref{lem:sobemb}, we obtain
\begin{align*}
\|F_1(t)\|_{L_x^r(|x|\geq
R_1)}\lesssim\|F_1\|_{L_x^2}^\frac{d-2}d\|F_1\|_{L_x^\infty(|x|\geq
R_1)}^\frac2d \lesssim R_1^{-\frac{d-1}d}.
\end{align*}
Therefore, by taking
$R_1^{\frac{d-2}2+\frac{d-1}d}=\epsilon^{-(\frac{d-2}2-\alpha_1)\theta}$
for $p\geq2$, and
$R_1^{\frac{pd-(d+2)}2+\frac{d-1}d}=\epsilon^{-(\frac{d}2+\alpha_1)(p-1)\theta+\theta}$
for $p<2$, we get by H\"older's inequality and \eqref{equ:l4l6est}
\begin{align}\label{equ:f1est}
\|F_1\|_{L_t^4([T,\infty),L_x^r)}\lesssim\|F_1\|_{L_t^\infty([T,\infty),L_x^r)}^\frac12
\|F_1\|_{L_t^2([T,\infty),L_x^r)}^\frac12 \lesssim\epsilon^\beta,
\end{align}
where
\begin{equation*}
  \beta=\begin{cases}
  \frac{2(d-1)}{d^2-2}(\frac{d-2}2-\alpha_1)\theta\quad\text{if}\quad
  p\geq2\\
  \frac{2(d-1)}{pd+d-4}(1-(\frac{d}2+\alpha_1)(p-1))\theta\quad\text{if}\quad p<2.
  \end{cases}
\end{equation*}

{\bf Estimate the term $F_2(t)$.} First, by \eqref{equ:masslim}, we
may choose $T>T_0$
\begin{equation}\label{equ:assumest}
\int \chi_R(x)|u(T,x)|^2\;dx\leq \epsilon^2,
\end{equation}
where $\chi_R(x)\in C_c^\infty(\R^d)$ and
\begin{equation*}
\chi_R(x)=\begin{cases}
1\quad \text{if}\quad |x|\leq R,\\
0\quad \text{if}\quad |x|\geq 2R.
\end{cases}
\end{equation*}
On the other hand, combining the identity $\pa_t|u|^2=-2\nabla\cdot{\rm Im}(\bar{u}\nabla u)$ and integration by parts, H\"older's inequality, we obtain
\begin{equation*}
\Big|\pa_t\int \chi_R(x)|u(t,x)|^2\;dx\Big|\lesssim\frac1R.
\end{equation*}
Hence, choosing $R\gg\epsilon^{-2-\theta}$, we get by
\eqref{equ:assumest}
\begin{equation}\label{equ:smallcut}
\|\chi_Ru\|_{L_t^\infty L_x^2(I_2\times\R^d)}\lesssim \epsilon.
\end{equation}
And so, by H\"older's inequality, Sobolev embedding and   Lemma
\ref{lem:sobemb}, we have for $q=\frac{2(d+2)}{d}$
\begin{align*}
\|u\|_{L_{t,x}^{q}(I_2\times\R^d)}\lesssim&\epsilon^{-\frac{\theta}{q}}\|u\|_{L_t^\infty(I_2,L_x^{q})}\\
\lesssim&\epsilon^{-\frac{\theta}{q}}\Big(\|\chi_Ru\|_{L_t^\infty(I_2,L_x^2)}^\frac2{d+2}\|u\|_{L_t^\infty
L_x^\frac{2d}{d-2}}^\frac{d}{d+2}
+\big\|(1-\chi_R)u\big\|_{L_{t,x}^\infty}^\frac2{d+2}\|u\|_{L_t^\infty L_x^2}^\frac{d}{d+2}\Big)\\
\lesssim&\epsilon^{-\frac{\theta}q}\big(\epsilon^\frac2{d+2}+R^{-\frac{d-2}{d+2}}\big)\lesssim
\epsilon^{\frac2{d+2}-\frac{\theta}q}.
\end{align*}
On the other hand, using Strichartz estimate and continuous
argument, we have
$$\|\la^\frac14u\|_{L_t^\frac{2(d+2)}{d-2}(I_2,L^\frac{2d(d+2)}{d^2+4})}^\frac{2(d+2)}{d-2}+\|u\|_{L_{t,x}^\frac{2(d+2)}{d-2}(I_2\times\R^d)}^\frac{2(d+2)}
{d-2}\lesssim 1+|I_2|.$$
 Thus, we use Sobolev embedding, Strichartz
estimate, equivalence of Sobolev spaces (Lemma
\ref{pro:equivsobolev}) to get
\begin{align}\nonumber
\|F_2\|_{L_t^{4}([T,\infty),L_x^{r})}\lesssim&\big\||\nabla|^\frac12F_2\big\|_{L_t^4((T,\infty),L_x^\frac{2d}{d-1})}\\\nonumber
\lesssim&\big\|\la^{\frac14}(|u|^{p-1}u)\big\|_{L_t^2(I_2,L_x^\frac{2d}{d+2})}\\\nonumber
\lesssim&\|u\|_{L_{t,x}^{\frac{d+2}2(p-1)}(I_2\times\R^d)}^{p-1}\|\la^\frac14u\|_{L_t^\frac{2(d+2)}{d-2}(I_2,L^\frac{2d(d+2)}{d^2+4})}\\\nonumber
\lesssim&|I_2|^\frac{d-2}{2(d+2)}\Big(\|u\|_{L_{t,x}^q(I_2\times\R^d)}^{1-s_c}
\|u\|_{L_{t,x}^\frac{2(d+2)}{d-2}(I_2\times\R^d)}^{s_c}\Big)^{p-1}\\\nonumber
\lesssim&|I_2|^{\frac{d-2}{2(d+2)}+\frac{d-2}{2(d+2)}s_c(p-1)}\epsilon^{\frac{4-\theta
d}{2(d+2)}(1-s_c)(p-1)}\\\nonumber \lesssim&
\epsilon^{\frac{4-\theta
d}{2(d+2)}(1-s_c)(p-1)-\frac{d-2}{2(d+2)}(1+s_c(p-1))\theta}\\\label{equ:plar}
\lesssim&\epsilon^{\frac{4-\theta d}{4(d+2)}(1-s_c)(p-1)}
\end{align}
by taking $\frac{4-\theta
d}{4(d+2)}(1-s_c)(p-1)=\frac{d-2}{2(d+2)}(1+s_c(p-1))\theta$. This
together with \eqref{equ:duhamp}, \eqref{equ:smallline}, and
\eqref{equ:f1est}  yields that
\begin{equation*}
\|e^{i(t-T)\la}u(T)\|_{L_t^4([T,\infty),L_x^\frac{2d}{d-2}(\R^d))}\lesssim
\epsilon+\epsilon^\beta+\epsilon^{\frac{4-\theta
d}{4(d+2)}(1-s_c)(p-1)}.
\end{equation*}
And so \eqref{equ:redfur} follows. Therefore, we conclude the proof
of Lemma \ref{lem:scatcri}.

\end{proof}

\subsection{Virial identities}\label{S:virial}  In this section, we recall some standard virial-type identities.
 Given a weight $w:\R^d\to\R$ and a solution $u$ to \eqref{equ1.1}, we define
\[
V(t;w) := \int |u(t,x)|^2 w(x)\,dx.
\]
Using \eqref{equ1.1}, one finds
\begin{align}\label{virial}
& \partial_t V(t;w) = \int 2\Im \bar u \nabla u \cdot \nabla w \,dx, \\\nonumber
& \partial_{tt}V(t;w) = \int (-\Delta\Delta w)|u|^2 + 4\Re \bar u_j u_k w_{jk} + 4|u|^2 \tfrac{ax}{|x|^4}\cdot \nabla w - \tfrac{2(p-1)}{p+1}|u|^{p+1} \Delta w\,dx.
\end{align}

The standard virial identity makes use of $w(x)=|x|^2$.

\begin{lemma}[Standard virial identity]\label{L:virial0} Let $u$ be a solution to \eqref{equ1.1}.  Then
\[
\partial_{tt} V(t;|x|^2) = 8\Bigl[\|u(t)\|_{\dot H_a^1}^2 - \tfrac{d(p-1)}{2(p+1)}\|u(t)\|_{L_x^{p+1}}^{p+1}\Bigr].
\]
\end{lemma}

In general, we do not work with solutions for which $V(t;|x|^2)$ is finite. Thus, we need a truncated version of the virial identity (cf. \cite{OT}, for example). For $R>1$, we define $w_R(x)$ to be a smooth, non-negative radial function satisfying
\begin{equation}\label{phi-virial}
w_R(x)=\begin{cases} |x|^2 & |x|\leq \frac{R}2 \\ R|x| & |x|>R,\end{cases}
\end{equation}
with
\begin{equation}\label{equ:wr}
\pa_rw_R\geq0,~\pa_r^2w_R\geq0,~|\pa^\alpha w_R(x)|\lesssim_\alpha R|x|^{-|\alpha|+1},~|\alpha|\geq1.
\end{equation}

In this case, we use \eqref{virial} to deduce the following:

\begin{lemma}[Truncated virial identity]\label{L:virial} Let $u$ be a radial solution to \eqref{equ1.1} and let $R>1$. Then
\begin{align*}
\nonumber \partial_{tt}& V(t;w_R)\\
  &= 8\int_{|x|\leq\frac{R}2}\Bigl[|\nabla u(t)|^2+a\tfrac{|u|^2}{|x|^2} - \tfrac{d(p-1)}{2(p+1)} |u(t)|^{p+1}\Bigr]\;dx  \\
& \quad  + \int_{|x|>R} \Bigl[4aR\tfrac{|u|^2}{|x|^3}-\tfrac{2(d-1)(p-1)}{p+1}\tfrac{R}{|x|}|u|^{p+1}+\frac{4R}{|x|}(|\nabla u|^2-|\pa_ru|^2)\Bigr]\;dx \\
&\quad  +\int_{\frac{R}2\leq|x|\leq R}\Bigl[4{\rm Re}\pa_{jk}w_R\bar{u}_j\pa_ku+O\bigl(  \tfrac{R}{|x|}|u|^{p+1} + \tfrac{R}{|x|^3}|u|^2\bigr)\Bigr]\;dx.
\end{align*}
Furthermore, by \eqref{equ:wr}, we have that $$\int_{\frac{R}2\leq|x|\leq R}\Bigl[4{\rm Re}\pa_{jk}w_R\bar{u}_j\pa_ku\Bigr]\;dx\geq 0.$$
\end{lemma}

\subsection{Proof of Theorem \ref{T:radial}}

By the scattering criterion (Lemma \ref{lem:scatcri}) and H\"older's inequality, Theorem \ref{T:radial} follows from the following lemma.

\begin{lemma}\label{lem:poste}
There exists a sequence of times $t_n\to\infty$ and a sequence of radii $R_n\to\infty$ such that
\begin{equation}\label{equ:pot}
\lim_{n\to\infty}\int_{|x|\leq R_n}|u(t_n,x)|^{p+1}\;dx=0.
\end{equation}
\end{lemma}

It is easy to see that the above lemma can be derived by the following proposition (choosing $T$ sufficiently large and $R=\max\{T^{1/3}, T^{1/p}\}$).

\begin{proposition}[Morawetz estimate]\label{prop:mores}
Let $T>0.$ For $R=R(\delta,M(u),Q_a)$ sufficiently large, we have
\begin{equation}\label{equ:potensmal}
\frac1{T}\int_0^T\int_{|x|\leq R}|u(t,x)|^{p+1}\;dx\;dt\lesssim\frac{R}{T}+\frac1{R^2}+\frac1{R^{p-1}}.
\end{equation}
\end{proposition}

\begin{proof}
First, by Lemma \ref{L:virial}, we have
\begin{align}
\nonumber \partial_{tt}& V(t;w_R)\\
\label{equ:mainter1}
  &= 8\int_{|x|\leq\frac{R}2}\Bigl[|\nabla u(t)|^2+a\tfrac{|u|^2}{|x|^2} - \tfrac{d(p-1)}{2(p+1)} |u(t)|^{p+1}\Bigr]\;dx  \\
\label{equ:erro11}
& \quad  + \int_{|x|>R} \Bigl[4aR\tfrac{|u|^2}{|x|^3}-\tfrac{2(d-1)(p-1)}{p+1}\tfrac{R}{|x|}|u|^{p+1}+\frac{4R}{|x|}(|\nabla u|^2-|\pa_ru|^2)\Bigr]\;dx \\
\label{equ:erro21} &\quad  +\int_{\frac{R}2\leq|x|\leq R}\Bigl[4{\rm
Re}\pa_{jk}w_R\bar{u}_j\pa_ku+O\bigl(  \tfrac{R}{|x|}|u|^{p+1} +
\tfrac{R}{|x|^3}|u|^2\bigr)\Bigr]\;dx.
\end{align}

We define $\chi$ to be a smooth cutoff to the set $\{|x|\leq1\}$ and set $\chi_R(x)=\chi(x/R)$. Note that
\begin{equation}\label{equ:iden}
\int\chi_R^2|\nabla u|^2\;dx=\int\Big[|\nabla(\chi_Ru)|^2+\chi_R\Delta(\chi_R)|u|^2\Big]\;dx,
\end{equation}
we get
\begin{align*}
\eqref{equ:mainter1}=&8\int\chi_R^2\Bigl[|\nabla u(t)|^2+a\tfrac{|u|^2}{|x|^2} - \tfrac{d(p-1)}{2(p+1)} |u(t)|^{p+1}\Bigr]\;dx  \\
&+8\int(1-\chi_R^2)\Bigl[|\nabla u(t)|^2+a\tfrac{|u|^2}{|x|^2} - \tfrac{d(p-1)}{2(p+1)} |u(t)|^{p+1}\Bigr]\;dx  \\
=&8\Big[\big\|\chi_Ru\big\|_{\dot{H}^1_a}^2-\tfrac{d(p-1)}{2(p+1)}\|\chi_Ru\|_{L_x^{p+1}}^{p+1}\Big]+8\int(1-\chi_R^2)|\nabla u(t)|^2\;dx\\
&+\int O\big(\tfrac{|u|^2}{R^2}\big)\;dx+C\int_{|x|\geq R}|u|^{p+1}\;dx.
\end{align*}

Next, we claim that there exists $c>0$ such that
\begin{equation}\label{equ:mainlarg}
\big\|\chi_Ru\big\|_{\dot{H}^1_a}^2-\tfrac{d(p-1)}{2(p+1)}\|\chi_Ru\|_{L_x^{p+1}}^{p+1}\geq
c\|\chi_Ru\|_{L_x^{p+1}}^{p+1}.
\end{equation}
Indeed, by \eqref{equ:iden}, we have
$$\big\|\chi_Ru\big\|_{\dot{H}^1_a}^2\leq \|u\|_{\dot{H}^1_x}^2+O\big(\tfrac{M(u)}{R^2}\big),$$
and $\|\chi_Ru\|_{L_x^2}\leq \|u\|_{L_x^2}$. Then, \eqref{equ:mainlarg} follows by the same argument as Proposition \ref{P:coercive} (ii).

Now, applying the fundamental theorem of calculus on an interval $[0,T]$, discarding the positive terms, using \eqref{equ:mainter1}-\eqref{equ:mainlarg}, we obtain
$$\int_0^T\int_{\R^d}|\chi_Ru|^{p+1}\;dx\lesssim\sup_{t\in[0,T]}|V(t,w_R)|+\int_0^T\int_{|x|\geq R}|u(t,x)|^{p+1}\;dx\;dt+\frac{T}{R^2}M(u).$$

From H\"older's inequality, \eqref{virial} and \eqref{equ:uniformboun}, we get
$$\sup_{t\in\R}|\pa_tV(t,w_R)|\lesssim R.$$
On the other hand, by radial Sobolev embedding, one has
\begin{equation}
\int_{|x|\geq R}|u(t,x)|^{p+1}\;dx\lesssim \frac1{R^{p-1}}\|u\|_{L_t^\infty \dot{H}^1}^{p-1}M(u).
\end{equation}
Combining  the above together, we obtain
$$\frac1{T}\int_0^T\int_{|x|\leq R}|u(t,x)|^{p+1}\;dx\;dt\lesssim\frac{R}{T}+\frac1{R^2}+\frac1{R^{p-1}},$$
which is accepted. Hence we conclude the proof of Proposition \ref{prop:mores}.

Therefore, we complete the proof of Theorem \ref{T:radial}.

\end{proof}

\begin{center}

\end{center}

\end{document}